\documentclass[
	11pt,
	leqno,
]{amsart}

\usepackage{amsthm,amsmath,amsfonts,amssymb,graphicx,longtable,mathrsfs,mathtools,upref,tikz,physics}	
\usepackage[
	style=numeric-comp,	
	sorting=none,
	sortcites=true,
	doi=false,
	url=false,
	giveninits=true,
	hyperref]{biblatex}
\usepackage[linktoc=page]{hyperref}
\hypersetup{
	pdftitle={Burning games on strong path products}
	pdfauthor={Sally Ambrose, Evan Angelone, Jacob Chen, Daniel Ma, Arturo Ortiz San Miguel, Wraven Watanabe, Stephen Whitcomb, Shanghao Wu},
	colorlinks=true
}
\usepackage[noabbrev]{cleveref}

\renewcommand{\cref}[1]{\Cref{#1}} 
\usepackage[cal=rsfs,bb=txof]{mathalpha}
\usepackage[nodisplayskipstretch]{setspace}

\usepackage{pgfplots} 
\pgfplotsset{compat=1.18}
\usepackage{enumitem}

\usepackage{tikz}
\usepackage{xparse}
\usepackage{graphicx}
\usepackage[dvipsnames]{xcolor}
\usepackage{subcaption}
\usepackage{manfnt}

\NewDocumentCommand{\drawchessboard}{ m m m}{%
  \begin{tikzpicture}[scale=0.5]
    \foreach \x in {1,...,#2} {
      \foreach \y in {1,...,#2} {
        \pgfmathtruncatemacro{\iswhite}{mod(\x+\y,2)}
        \ifnum\iswhite=0
          \fill[lightgray] (\x-1,\y-1) rectangle (\x,\y); 
        \else
          \fill[white] (\x-1,\y-1) rectangle (\x,\y); 
        \fi
        \draw (\x-1,\y-1) rectangle (\x,\y); 
      }
    }

    \ifthenelse{\equal{#3}{true}}{
      \foreach \x in {1,...,#2} {
        \node[anchor=north] at (\x-0.5, -0.2) {\x};
      }
      \foreach \y in {1,...,#2} {
        \node[anchor=east] at (-0.2, \y-0.5) {\y};
      }
    }{}
    \begin{scope}[shift={(-0.5,-0.5)}]
      #1
    \end{scope}    
  \end{tikzpicture}}
\theoremstyle{plain}
	\newtheorem{theorem}{Theorem}[section]
\newtheorem{corollary}[theorem]{Corollary}
\newtheorem{lemma}[theorem]{Lemma}
\newtheorem{proposition}[theorem]{Proposition}
\newtheorem{conjecture}[theorem]{Conjecture}

\theoremstyle{remark}

\usepackage[foot]{amsaddr}	

\title[Burning games on strong path products]{Burning games on strong path products}

\author[]{Sally Ambrose$^{*}$}
\address{Northeastern University, Boston, Massachusetts, USA}
\email{\href{mailto:ambrose.sa@northeastern.edu}{ambrose.sa@northeastern.edu}}

\author[]{Evan Angelone$^{\dagger\ddagger}$}
\email{\href{mailto:griggs.e@northeastern.edu}{griggs.e@northeastern.edu}}

\author[]{Jacob Chen$^{*}$}
\email{\href{mailto:chen.jacob3@northeastern.edu}{chen.jacob3@northeastern.edu}}

\author[]{Daniel Ma$^{*}$}
\email{\href{mailto:ma.dan@northeastern.edu}{ma.dan@northeastern.edu}}

\author[]{Arturo Ortiz San Miguel$^{\dagger\ddagger}$}
\email{\href{mailto:ortizsanmiguel.a@northeastern.edu}{ortizsanmiguel.a@northeastern.edu}}

\author[]{Wraven Watanabe$^{*}$}
\email{\href{mailto:watanabe.so@northeastern.edu}{watanabe.so@northeastern.edu}}

\author[]{Stephen Whitcomb$^{*}$}
\email{\href{mailto:whitcomb.s@northeastern.edu}{whitcomb.s@northeastern.edu}}

\author[]{Shanghao Wu$^{*}$}
\email{\href{mailto:wu.shangh@northeastern.edu}{wu.shangh@northeastern.edu}}

\thanks{The authors were partially supported by NSF grant DMS-1645877.}
\thanks{$^{*}$Undergraduate mentee.}
\thanks{$^{\dagger}$PhD student mentor.}
\thanks{$^{\ddagger}$Corresponding author.}


\begin{document}
	\begin{abstract}
	Graph burning and cooling are diffusion processes where burned (or cooled) vertices spread to neighbors, with new sources added at discrete time steps. The burning number~$b(G)$ is the minimum time to burn graph~$G$; the cooling number~$\mathrm{CL}(G)$ is the maximum time to delay complete cooling.

	We prove that the burning number of~$d$-fold strong products of paths, which generalize king graphs and model radial contagion spread in~$d$ dimensions, can be bounded by reducing the problem to geometric tiling. We obtain sharp bounds using an Euler-Maclaurin formula when certain number-theoretic conditions hold. We also determine the cooling number for the~$d$-fold strong product of paths~$P_n$.

	Lastly, we establish results for $k$-liminal burning. We determine a lower bound for~$k^*$, the minimum integer~$k$ such that the~$k$-step burning number equals the standard burning number, for certain path graphs.
\end{abstract}
	\maketitle

	\section{Introduction}\label{sec:Introduction (Liminal Burning)}

	\subsection{Rules and notation}\label{sub: Rules & Notation}

	Graph burning is a deterministic process that models the spread of information or contagion across a given graph. The behavior of how information, or a virus, spreads through a population is a significant subject which is studied in many capacities, including diffusion models in graph theory. 

	In graph burning, each vertex in a graph exists in a binary state of \emph{burned} or \emph{unburned}. To begin this process, we start with a graph in which every vertex is unburned. The player selects an unburned vertex to burn, which we call a \emph{source}. This concludes round one. At the beginning of each subsequent round, any burned vertex will spread to its neighbors, causing them to burn as well. We call the process of spreading \emph{propagation}. The player then chooses an unburned vertex to burn and the turn ends. This continues until there are no more unburned vertices to select, and the graph is burned.
	In modeling how burning spreads across a graph, a natural question arises: how quickly can a graph be burned? To answer this, we look to the \emph{burning number}~$b(G)$, which is the minimum number of turns needed to burn a graph~$G = (V, E)$. 

	The existence of graph burning lends itself to a complementary process, cooling. Similar to burning, we begin with a completely uncooled graph, and over discrete time steps, cooled vertices spread this condition to their neighbors. The process functions identically to burning, with one key difference. Rather than attempting to minimize the turns needed to burn a graph, the player seeks to maximize the number of turns taken to cool a graph,~$G$. We call this number the \emph{cooling number}, denoted~$\mathrm{CL}(G)$.

	In \cref{sec: Kings and Strong Product of Paths}, we consider the burning and cooling games on strong products of paths. We denote the~$d$-fold strong product of a path of length~$n$ as~$P_n^{\boxtimes d}$. In this family of graphs, propagation expands radially, which is a natural way to model the spread of fires and infections in an arbitrary number of dimensions. We completely determine the cooling numbers for all~$n,d$ and give upper bounds for the burning numbers that are sharp under certain number theoretic conditions.

	In \cref{sec: Paths}, we consider the \emph{$k$-liminal burning game}, a new variant of the burning game involving two players, an \emph{arsonist} and a \emph{saboteur}, introduced in \cite{bonato2025between}. The game is played on a finite connected graph where vertices are independently classified as \emph{burned} or \emph{unburned}, and as \emph{revealed} or \emph{unrevealed}. On graph~$G$ with \emph{liminal number}~$k$, the game proceeds as follows:
	\begin{enumerate}
		\item The saboteur reveals a set of~$k$ vertices in~$G$. These vertices remain revealed for the entirety of the game.
		\item The arsonist burns a single revealed \emph{source} vertex~$v$.
		\item Each subsequent turn begins by burning all vertices adjacent to an already burned vertex, as in the standard burning game. The saboteur then reveals~$k$ additional vertices, and the arsonist selects another vertex to burn.
		\item The game proceeds in this fashion until all vertices in~$G$ have been burned, at which point we say~$G$ is burned.
	\end{enumerate}
	The saboteur plays to maximize the number of turns needed to burn~$G$ while the arsonist plays to minimize it. We are interested in the \emph{$k$-liminal burning number} $b_k(G)$ of turns it takes to burn~$G$ given the liminal number~$k$, provided both the saboteur and the arsonist are playing optimally.

	We note that~$b_{|V|}(G) = b(G)$, since the arsonist is able to burn vertices in the same order as the standard burning sequence of~$G$, thus burning~$G$ in as few turns as possible. Similarly,~$b_1(G) = \mathrm{CL}(G)$, since the saboteur reveals vertices in the same order as the cooling sequence of~$G$ in order to burn~$G$ in as many turns as possible. Likewise, it is evident that
	\begin{align*}
		\mathrm{CL}(G) = b_1(G) \leq b_2(G) \leq \cdots \leq b_{|V|-1} \leq b_{|V|} = b(G),
	\end{align*}
	where the burning and cooling numbers of~$G$ provide rigid bounds for~$b_k(G)$, whose behavior can be viewed in \cref{liminal plot}. This implies the existence of some minimal~$k \leq |V|$ such that~$b_k(G)=b(G)$, which we denote as~$k^*$. In particular, in \cref{sec: Paths}, we give bounds for~$k^*$ for~$G = P_n$.
	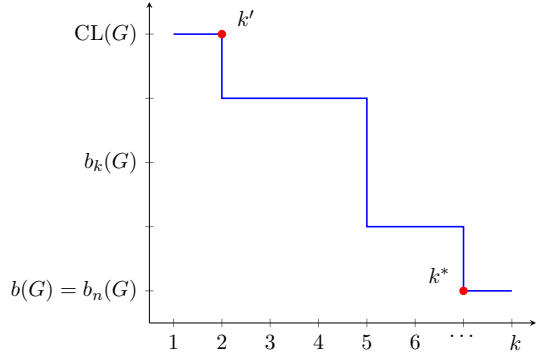
\begin{figure}[htbp!]
		\centering
		\begin{tikzpicture}[scale=0.75]
			\begin{axis}[
				xmin=0.5, xmax=8.5,
				ymin=1.5, ymax=6.5,
				axis lines=center,
				xtick={1,2,3,4,5,6,7,8},
				xticklabels={1,2,3,4,5,6,$\dots$,~$k$},
				ytick={2,3,4,5,6},
				yticklabels={$b(G)=b_n(G)$, ,$b_k(G)$ , ,~$\mathrm{CL}(G)$},
				samples=100
			]
			
			\addplot[blue, thick, const plot mark right] coordinates {
				(1,6) (2,6) (3,5) (4,5) (5,5) (6,3) (7,3) (8,2)
			};
			
			\node[above] at (axis cs:2.5,6) {$k'$};
			\addplot[red] coordinates {(2,6)} node[circle,fill,inner sep=1.5pt]{};
			
			\node[above] at (axis cs:6.5,2) {$k^*$};
			\addplot[red] coordinates {(7,2)} node[circle,fill,inner sep=1.5pt]{};
			\end{axis}
		\end{tikzpicture}
		\caption{Liminal burning numbers of~$G$}
		\label{liminal plot}
	\end{figure}
	\subsection{Discussion and results}\label{sub: Discussion}

	Burning has been studied both generally and for specific families of graphs. One of the long-standing conjectures is the burning number conjecture, which states that for a graph~$G$ with~$n$ vertices, we have~$b(G) \leq \lceil \sqrt{n} \rceil$. This bound is sharp for paths. The best known bound is
	\begin{align*}
		b(G) \leq \left\lceil{\frac{-3 + \sqrt{24n + 33}}{4}}\right\rceil,
	\end{align*}
	and may be found in \cite{land2016upper}.
	It is also known that graph burning reduces to burning on trees via the following due to \cite{bonato2015burn}:
	\begin{align*}
		b(G) = \min\{b(T) \mid \text{$T$ a spanning tree of~$G$}\}.
	\end{align*}
	The relationship between the burning number and different graph products has received attention as well. In \cite{mitsche2018burning}, it was shown that for connected graphs~$G$ and~$H$, we have
	\begin{align}
		\max\{b(G), b(H)\} &\leq b(G \boxtimes H) \label{eq:strong product bound}\\
        &\leq b(G \,\square\, H) \nonumber\\
        &\leq \min\{b(G) + \mathrm{rad}(H), b(H) + \mathrm{rad}(G)\}. \nonumber
	\end{align}

	Cooling has received less attention in the literature. Several known results on~$\mathrm{CL}(G)$ are mentioned in \cite{bonato2025between}. In Proposition \ref{prop:strong cooling} of this paper, we prove that $\mathrm{CL}(P_n^{\boxtimes d}) = n$ for any~$n, d$. In contrast, burning numbers are much harder to compute, but we give bounds that are sharp when the graph can be tiled.
	As such, we use the paradigm of viewing burning~$G$ as a tiling problem throughout \cref{sec: Kings and Strong Product of Paths}, that is, if~$G$ can be tiled appropriately, then~$b(G)$is exactly the number of tiles used. For~$G= P_n^{\boxtimes d}$, the tiles are~$d$-dimensional cubes with decreasing odd side lengths. For small~$d$ the optimal number of burned vertices on~$P_n^{\boxtimes d}$ after~$m$ turns can be determined by using number-theoretic identities. For larger~$d$, we refer to algebraic methods of deriving Euler-Macluarin formulas using \emph{SI-interpolators} as in \cite{fischer2022algebraic} to determine the largest~$x$ such that~$x$ of these tiles can pack~$[n]^d$. The idea is to extend the identity into a rational polynomial~$\overline{g}_d$ and find the largest root~$x^*$ of~$n^d - \overline{g}_d(x)$ so that~$b(P_n^{\boxtimes d}) \geq \lfloor x^* \rfloor$. In other words, if~$\lfloor x^* \rfloor$ disjoint~$d$-cubes (whose side lengths are decreasing odd numbers) fit into~$[n]^d$, then~$b(P_n^{\boxtimes d})$ must be at least~$x^*$. Our methods culminate in the following result:
	\begin{theorem}\label{thm:burning strong}
		For fixed~$n, d \in \mathbb{N}$, let~$x^*$ be the largest real root of
		\begin{align*}
			q(x) \coloneq n^d &- \frac{(2x - 1)^{d + 1} - 1}{2(d + 1)} \\ &- \sum_{k = 0}^d{\frac{B_{k + 1}}{(k + 1)!} \frac{2^k d!}{(d - k)!} \left[(2x - 1)^{d - k} + (-1)^{k} \right]},
		\end{align*}
		where~$B_i$ is the~$i$-th Bernoulli number. Then
		\begin{align*}
			b(P_n^{\boxtimes d}) \geq \lfloor x^* \rfloor,
		\end{align*}
		with equality if and only if~$x^* \in \mathbb{N}$.
	\end{theorem}
	\noindent It is simple to verify that $x^*$ is $O(n^{\frac{d}{d + 1}})$, and that this bound is strictly better than the bound in Equation~\ref{eq:strong product bound}. That these bounds are asymptotically sharp is the topic of Conjecture~\ref{conj}. Corollaries \ref{kings cor} and \ref{3d cor} give a closed form for this~$x^*$ for~$d = 2, 3$.

	Next, in \cref{sec: Paths}, we focus on liminal burning, which recently emerged as a generalization of both burning and cooling. While similar variants exist (see \cite{chiarelli2024burninggame} for an example), liminal burning is largely unexplored. We begin by building on the work of \cite{bonato2025between}, who originally introduced this game.

	We first consider one of the most straightforward settings for liminal burning, the path of length~$n$, denoted~$P_n$, with vertex set~$\{v_1, v_2, \dots, v_n \}$. For~$b_k(P_n)$, the authors in \cite{bonato2025between} give bounds:
	\begin{align*}
		\left\lfloor \frac{n}{k+1} \right\rfloor + \left\lfloor \frac{-1 + \sqrt{5 + 4k}}{2} \right\rfloor \leq b_k(P_n) \leq \left\lceil \frac{n}{k} \right\rceil + k - 1.
	\end{align*}
	To generate this lower bound, we assume the saboteur always reveals the set of~$k$ vertices nearest to~$v_1$ and that the arsonist burns the rightmost revealed vertex. To reach the upper bound, we partition~$P_n$ into~$\lceil n/k \rceil$-many sets with at most one set containing more than~$k$ elements, where the arsonist places a source in a distinct~$k$-set during each of the first~$\lceil n/k \rceil$ turns (for the full proof, see \cite{bonato2025between}). It is surprising that the~$k$-liminal burning numbers of~$P_n$ have not been determined. We hope to tighten the bounds for~$b_k(P_n)$ and derive bounds for~$k^*$ for~$G=P_n$ by not assigning a specific strategy to the players.
	
	Each source placed by the arsonist makes a \emph{tile}, which is made of the vertices that were burned by the propagation from the source. If possible, the arsonist's optimal strategy is to fill~$G$ with non-overlapping tiles. The saboteur, on the other hand, will reveal vertices that make such a tiling impossible. For the path~$P_{n^2}$, the tiles are sub-paths of lengths~$1, 3, \dots, 2n - 1$, placed end-to-end. We show the following result for finding~$k^*$ for paths. Consider the polynomial
	\begin{align*}
		P(x,y) \coloneq \prod_{i = 1}^n (1 + x^{2i - 1} y) = \sum_{\ell, r} c_{\ell, r} \cdot x^\ell y^r,
	\end{align*}
	so that~$f(n, \ell) \coloneq \sum_{\ell, r}{c_{\ell, r} \cdot r! (n - 1 - r)!}$.
	\begin{theorem}\label{thm: k* on path graph}
		Fix~$n \geq 2$, and let~$k^*$ be the smallest~$k$ such that~$b_k(P_{n^2})$ and~$b(P_{n^2})$ are both~$n$. Then
		\begin{align*}
			k^* > n^2 - \sum_{\ell = 0}^n \mathbf{1}_+(f(n, \ell)).
		\end{align*}
		Here, we understand that~$\mathbf{1}_+(x)$ is~$1$ for~$x > 0$, and~$0$ otherwise.
	\end{theorem}
	\section{Burning and cooling strong products of paths}\label{sec: Kings and Strong Product of Paths}

	\noindent We can represent~$P_n^{\boxtimes d}$ by placing vertices on a grid~$V = [n]^d$ and connecting vertices~$v = (v_1, v_2, \dots, v_d)$ and~$w = (w_1, w_2, \dots, w_d)$ when~$|v_i - w_i| \leq 1$ for all~$i \in [d]$. In this family of products, the propagation expands radially from the source vertices. In general, we have that~$\mathrm{diam}(P_n^{\boxtimes d}) = n$, hence the cooling number is bounded by~$\mathrm{CL}(P_n^{\boxtimes d}) \leq n$. We show this bound is sharp.

	\begin{proposition}\label{prop:strong cooling}
		For every ~$n, d > 1$,
		\begin{align*}
			\mathrm{CL}(P_n^{\boxtimes d}) = n.
		\end{align*}
	\end{proposition}
	
	\begin{proof}
		We give an explicit cooling sequence. Pick a 2-face of the~$d$-cube defined by the grid. The subgraph defined by this 2-face is a king graph. Color the vertices like a chessboard, with the bottom left vertex black. Then, the cooling sequence consists of the black vertices on the bottom row from left to right and then the white vertices on the right-most column from bottom to top. See Figure~\ref{fig:cooling sequence for strong path product} for examples of this sequences. The graph is cooled at the beginning of the~$n$-th turn.
	\end{proof}
	
	\begin{figure}[htbp!]
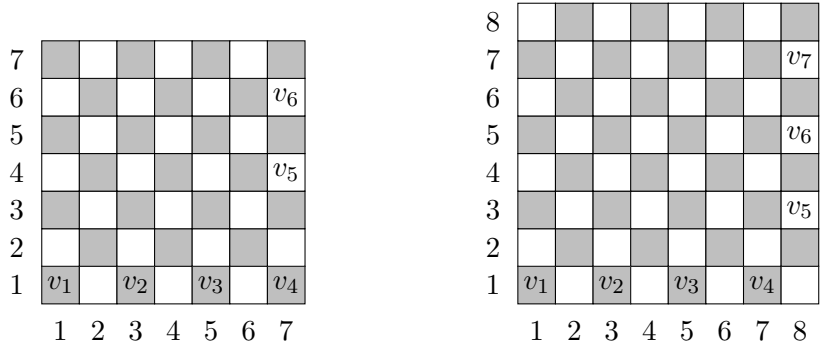

	\centering
	\begin{subfigure}[t]{0.48\textwidth}
		\centering
		\drawchessboard{
			\node at (1,1) {$v_1$};
			\node at (3,1) {$v_2$};
			\node at (5,1) {$v_3$};
			\node at (7,1) {$v_4$};
			\node at (7,4) {$v_5$};
			\node at (7,6) {$v_6$};
		}{7}{true}
	\end{subfigure}%
	\hfill
	\begin{subfigure}[t]{0.48\textwidth}
		\centering
		\drawchessboard{
			\node at (1,1) {$v_1$};
			\node at (3,1) {$v_2$};
			\node at (5,1) {$v_3$};
			\node at (7,1) {$v_4$};
			\node at (8,3) {$v_5$};
			\node at (8,5) {$v_6$};
			\node at (8,7) {$v_7$};
		}{8}{true}
	\end{subfigure}
	\caption{Cooling sequences for~$n=7,8$ in a 2-face of cube defined by~$P_n^{\boxtimes d}$.}
	\label{fig:cooling sequence for strong path product}
\end{figure}
	\subsection{\texorpdfstring{Using Euler-Maclaurin formulas to bound $b(P_n^{\boxtimes d})$}{Using Euler-Maclaurin formulas to bound the burning number for strong products of paths}}\label{sub:strong_product_bound_for_arbitrary_d}

	For burning, the optimal strategy can be described as a tiling problem where the tiles are~$d$-dimensional cubes with odd side lengths. We rely on number-theoretic identities of the form
	\begin{align}\label{eq:Euler-Maclaurin intro}
		1^d + 3^d + \dots + (2n - 1)^d = g(n,d),
	\end{align}
	for some closed form expression~$g(n,d)$. For fixed~$d$, we may write~$g(n, d)$ as a function~$g_d \colon \mathbb{N} \longrightarrow \mathbb{N}$ of a single variable, which extends to a rational polynomial~$\overline{g}_d$. Recognizing the left side of Equation \ref{eq:Euler-Maclaurin intro} as the sum of the function values~$f(x) = (2x - 1)^{d}$ over the integer points of the~$1$-polytope (line segment)~$P = [1, n]$, we can refer to modern treatments of Euler-Maclaurin formulas over integral polytopes to give a bound for~$b(P_n^{\boxtimes d})$.

	We let~$\varSigma_P$ denote the normal fan to~$P$, which consists of the cones~$\sigma_F$ dual to the set of directions one can venture from a given face~$F$ of~$P$ (denoted by~$F \leq P$) while remaining in~$P$. The construction given in Theorem 1 of \cite{fischer2022algebraic} allows us to write
	\begin{align*}
		\sum_{x \in P \cap \mathbb{Z}}{f(x)} &= \sum_{F \leq P}\int_{F}{\mu(\sigma_F)(-\partial_x) \circ f} \dd{m_F},
	\end{align*}
	where~$\mu(\sigma_F)$ is the \emph{SI-interpolator} constructed in \cite{fischer2022algebraic}, and~$\dd{m_F}$ is the relative Lebesgue measure on the affine span~$\mathrm{aff}(F)$ normalized with respect to the lattice~$\mathrm{aff}(F) \cap \mathbb{Z}$. For our purposes, this SI-interpolator is a differential operator applied to the polynomial~$f$, given below.

	The normal fan~$\varSigma_P$ consists of exactly three cones---two~$1$-dimensional cones~$\sigma_{\{1\}} = \mathrm{Cone}(e_1)$ and~$\sigma_{\{n\}} = \mathrm{Cone}(-e_1)$, and a single~$0$-dimensional cone~$\sigma_P = \{0\}$. As per Example 7.1 in \cite{fischer2022algebraic}, the formulas for the SI-interpolators of these particular cones are given by
	\begin{align*} 
		\mu(\sigma_P)(-\partial_x) &= 1,\\
		\mu(\sigma_{\{1\}})(-\partial_x) &= \left.\frac{\mathrm{td}(e_1) - 1}{e_1}\right|_{-\partial_x} = \sum_{k \geq 0}{(-1)^k \frac{B_{k + 1}}{(k + 1)!} \dv[k]{}{x}}, \\ 
		\mu(\sigma_{\{n\}})(-\partial_x) &= \left.\frac{\mathrm{td}(-e_1) - 1}{-e_1}\right|_{-\partial_x} = \sum_{k \geq 0}{\frac{B_{k + 1}}{(k + 1)!} \dv[k]{}{x}}.
	\end{align*}
	where here we understand~$\mathrm{td}(z)$ to be the \emph{Todd power series},
	\begin{align*}
		\mathrm{td}(z) = \frac{z}{1 - e^{-z}} = \sum_{k \geq 0}{\frac{B_k}{k!} z^k},
	\end{align*}
	whose coefficients~$B_k$ are the \emph{Bernoulli numbers}. As~$f$ is a degree-$d$ polynomial, it is sufficient to compute~$\mu(\sigma_{\{i\}})$ up to the degree-$d$ term for~$i = 1, n$. Putting these ingredients together, computing the~$k$-th derivative of~$f$ gives the following lemma.
	\begin{lemma}\label{lem:Evan's Lemma}
		Fix~$d$ and let~$f(x) = (2x - 1)^d$. Letting
		\begin{align*}
			g_d(n) = \sum_{k = 1}^{n}{(2k - 1)^d}
		\end{align*}
		as in Equation \ref{eq:Euler-Maclaurin intro} extend to a function~$\overline{g}_d$ on the reals, we have
		\begin{align*}
			\overline{g}_d(x) = \frac{(2x - 1)^{d + 1} - 1}{2(d + 1)} + \sum_{k = 0}^d{\frac{B_{k + 1}}{(k + 1)!} \frac{2^k d!}{(d - k)!} \left[(2x - 1)^{d - k} + (-1)^{k} \right]},
		\end{align*}
		where~$B_k$ is the~$k$-th Bernoulli number with~$B_1 = +1/2$.
	\end{lemma}
	\begin{proof}[Proof of Theorem \ref{thm:burning strong}]
		Note that~$\overline{g}_d$ is strictly increasing on its domain. Hence, for fixed~$n,d > 1$, the largest~$x$ such that~$\overline{g}_{d}(x) \leq n^d$ is the largest root~$x^*$ of the rational polynomial~$q(x) = n^d - \overline{g}_d(x)$. By Lemma \ref{lem:Evan's Lemma}, we can pack~$[n]^d$ with~$\lfloor x^* \rfloor$-many tiles. The packing is a tiling if and only if~$x^*\in \mathbb{N}$.
	\end{proof}


	For small $d$, one may verify that $g_d$ corresponds with the well-known identities,
	\begin{align*}
		\sum_{k = 1}^n (2k - 1) = n^2 \quad \text{and} \quad \sum_{k = 1}^n (2k - 1)^2 = \frac{n(2n - 1)(2n + 1)}{3}.
	\end{align*}
	From Theorem \ref{thm:burning strong}, we recover the known value for $b(P_n)$. For general king graphs $\mathcal{K}_n = P_n \boxtimes P_n$, we get the following lower bound
	\begin{corollary}\label{kings cor}
		For every $n > 1$, let $A_n = 27n^2 + \sqrt{3} \cdot \sqrt{243n^4 - 1}$. Then
		\begin{align*}
			b(P_n \boxtimes P_n) \geq \left\lfloor \frac{1}{2\sqrt[3]{3} \cdot \sqrt[3]{A_n}} + \frac{\sqrt[3]{A_n}}{2 \sqrt[3]{9}} \right\rfloor.
		\end{align*}
	\end{corollary}
	\begin{proof}
		Using Cardano's formula on the cubic \begin{align*}
			n^2 - \frac{x(2x-1)(2x+1)}{3},
		\end{align*}
		gives the proposed bound as the largest root.
	\end{proof}

	It turns out that Corollary \ref{kings cor} is never sharp. This is because it is impossible to tile a square with consecutive odd squares. Thus, we get a small improvement on the bound.

	\begin{lemma}\label{squares}
		For every $k > 1$, the number $k(2k+1)(2k-1)/3$ is not a perfect square.
	\end{lemma}
	\begin{proof}
		Note that for $a,b \in \mathbb{N}$ with $\gcd(a,b) = 1$, we have $ab = m^2$ if and only if $a,b$ are both squares. Also, note that 
		\begin{align*}
			\gcd(k, 2k+1) = \gcd(k, 2k-1) = \gcd(2k+1, 2k-1) = 1.
		\end{align*}
		Thus, 3 divides exactly one of $k, (2k+1), (2k-1)$. Set $a$ to be this term, and $b$ the product of the remaining two.
	\end{proof}
	
	\begin{corollary}\label{kings}
		For every $n > 1$, let $A_n$ be as in Corollary~\ref{kings cor}. Then
		\begin{align*}
			b(P_n \boxtimes P_n) \geq \left\lfloor \frac{1}{2\sqrt[3]{3} \cdot \sqrt[3]{A_n}} + \frac{\sqrt[3]{A_n}}{2 \sqrt[3]{9}} \right\rfloor + 1.
		\end{align*}
	\end{corollary}
	\begin{proof}
		For $n > 1$, note by Lemma \ref{squares} that the bound from Corollary \ref{kings cor} is not tight.
	\end{proof}
	\noindent Remembering that the king graph has $n^2$ vertices, this means that the bound is $O(m^{1/3})$ where $m$ is the total number of vertices, which we conjecture is asymptotically sharp. 
	
	Similarly, for $d = 3$, we get that $g_3(n)$ corresponds with the identity,
	\begin{align*}
		\sum_{k = 1}^n (2k - 1)^3 = 2n^4 - n^2,
	\end{align*}
	and so finding the largest root $m^*$ of $n^3 - 2x^4 + x^2$ gives the following result.
	\begin{corollary}\label{3d cor}
		For every $n > 1$,
		\begin{align*}
			b(P_n \boxtimes P_n \boxtimes P_n) \geq \left\lfloor \frac{1}{2} \sqrt{1 + \sqrt{1 + 8n^3}}\right\rfloor = \lfloor m^*\rfloor,
		\end{align*}
		with equality if and only if $m^* \in \mathbb{N}$.
	\end{corollary}
	Noting that this time there are $n^3$ vertices, the bound is $O(m^{1/4})$ where $m$ is the number of vertices. We conjecture the following.
	\begin{conjecture}\label{conj}
		For every $n, d \geq 1$,
		\begin{align*}
			b(P_n^{\boxtimes d}) = \Theta(n^{\frac{d}{d + 1}}).
		\end{align*}
	\end{conjecture}

	\section{Liminal burning on paths}\label{sec: Paths}

	\noindent By \cite{bonato2025between}, we know that if~$k=O(1)$ then~$b_k(P_n) = \Theta(n)$. In this section, we give exact values for certain~$k$. Theorem~\ref{thm:2-liminal burning path graph} gives an exact value for~$b_k(P_n)$ when~$k = 2$, while Theorem~\ref{thm: k* on path graph} deals with large~$k$.
	\begin{theorem}\label{thm:2-liminal burning path graph}
		For every~$n \geq 1$,
		\begin{align*}
			b_2(P_n) = \left\lceil \frac{n + 2}{3} \right\rceil.
		\end{align*}
	\end{theorem}

	\begin{proof}
		Let~$P_n$ be the path on the vertices~$V = \{v_1, \dots, v_n\}$ with~$v_i$ adjacent to~$v_{i + 1}$ for each~$i = 1, \dots, n - 1$. We propose a saboteur strategy for~$k = 2$ on~$P_n$, and claim that it is optimal. Our proposed strategy is simple: on the first turn, the saboteur reveals~$v_{1}$ and~$v_{2}$, and on turn~$i \geq 2$, the saboteur reveals the ``left-most'' unburned vertices after turn~$i - 1$. That is to say, if~$v_{1}, \dots, v_{j}$ are burned after turn~$i$, then the saboteur chooses to reveal~$v_{j + 1}$ and~$v_{j + 2}$.

		To play optimally against the saboteur's first pair of revealed vertices, the arsonist will choose to burn the revealed vertex whose neighborhood contains the largest number of unburned vertices. On the first turn, this means the arsonist will burn~$v_2$ so that both~$v_1$ and~$v_3$ burn at the start of the second turn. This justifies the saboteur's strategy on turns~$i \geq 2$.
		
		

		For this arsonist strategy, exactly one vertex is burned on the first turn, while exactly three vertices are burned for all subsequent turns (except for possibly the final turn~$b_2(P_n)$, which could burn one, two, or three vertices depending on~$n$) as the fire from the previous turn spreads.
		
	
	
		As such, the burning number depends on the final turn, which gives
		\begin{align*}
			b_2(P_n) = 
			\begin{cases}
				\frac{1}{3}(n - 1) + 1, & \text{if~$n \equiv 1 \bmod{3}$}; \\
				\frac{1}{3}(n - 2) + 2, & \text{if~$n \equiv 2 \bmod{3}$}; \\
				\frac{1}{3}(n - 3) + 2, & \text{if~$n \equiv 0 \bmod{3}$}.
			\end{cases}
		\end{align*}
		This is just~$b_2(P_n) = \lceil (n + 2)/3 \rceil$, so we're done.
	\end{proof}

	For general~$k$, we appeal to another tiling argument. It is easy to verify that~$b(P_{n^2}) = n$. Here, the arsonist’s optimal strategy may be 
	interpreted as a discrete tiling of the interval~$[1, n^2]$ by~$n$-many tiles of increasing odd lengths. If possible, the arsonist's optimal strategy is to perform such a tiling by picking the midpoints of each tiles as sources from largest to smallest. The saboteur will play in such a way that will prevent a tiling from being possible. Recall that~$k^*$ is the smallest~$k$ such that~$b_k(P_{n^2}) = b(P_{n^2}) = n$. Clearly,~$k^*$ is at least~$n^2 - m$, where~$m$ is the number of places the first tile may be placed. We prove Theorem \ref{thm: k* on path graph} by giving necessary conditions for this and considering the placement of the rest of the tiles.
	\begin{proof}[Proof of Theorem \ref{thm: k* on path graph}]
		We assume that the first vertex chosen by the arsonist to burn is the tile of largest length~$2n - 1$, whose left endpoint is at~$\ell$, and that it has been placed so that its midpoint is placed on vertex~$v_{\ell + n}$, for some~$\ell \in \{0, 1, \dots, (n - 1)^2 \}$. As such, the lengths of the remaining tiles are to be chosen from the set~$T \coloneq \{1, 3, \dots, 2n - 3\}$. The subinterval~$[1, 1 + \ell] \subset [1, n]$ to the ``left'' of the first tile must then be tiled by some subset~$S \subset T$, with the subinterval to the ``right'' of the first tile then tiled by~$T \setminus S$. Manifestly,~$\sum_{s}s = \ell$, where the sum is over all~$s \in S$.

		Consider the generating function
		\begin{align*}
			P(x, y) \coloneq \prod_{i = 1}^{n - 1}{(1 + x^{2i - 1} y)} = \sum_{\ell, r}{c_{\ell, r} \cdot x^{\ell} y^r}.
		\end{align*}
		In the above display, the coefficients~$c_{\ell, r}$ are exactly the number of~$r$-sets~$S \subset T$ which compose~$\ell$ (or, equivalently, tile~$[1, 1 + \ell]$ in a particular order). Thus the number of ways in which the arsonist can tile~$[n^2]$ after placing the first tile of length~$2n - 1$ is precisely
		\begin{align*}
			f(n, \ell) \coloneq \sum_{r = 0}^{n - 1}{c_{\ell, r} \cdot r! (n - 1 - r)!}.
		\end{align*}

		Let~$\mathbf{1}_+(x)$ be the indicator function given by~$1$ for~$x > 0$ and~$0$ otherwise. Since we assume the saboteur reveals at least as many vertices necessary to guarantee there exists at least one tiling of~$[n^2]$ after placing the initial tile of length~$2n - 1$, we get that
		\begin{align*}
			k^* > n^2 - \sum_{\ell = 0}^n \mathbf{1}_+(f(n, \ell)),
		\end{align*}
		so we're done.
	\end{proof}

	\section*{Acknowledgements}\label{sec:acknowledgements}

	This research was part of the Northeastern Summer Math Research Program (NSMRP), which is an REU. The authors are grateful for the Northeastern College of Science and Department of Mathematics for partially funding this research.

	\printbibliography
\end{document}